\documentclass[11pt,noinfoline]{article}
\RequirePackage{times}
\RequirePackage[OT1]{fontenc}
\RequirePackage[aop,amsthm,amsmath]{imsart}
\RequirePackage{amssymb,amscd,mathrsfs,eucal}
\usepackage[title]{appendix}

\startlocaldefs
\def\journal@id{~}
\def\journal@name{~}
\def\journal@url{~}

\newtheorem{thm}{Theorem}[section]
\newtheorem{cor}[thm]{Corollary}
\newtheorem{lem}[thm]{Lemma}
\newtheorem{prop}[thm]{Proposition}

\theoremstyle{definition}
\newtheorem{defn}[thm]{Definition}
\theoremstyle{remark}
\newtheorem{rem}[thm]{Remark}
\newtheorem*{rem*}{Remark}

\numberwithin{equation}{section}
\endlocaldefs

\DeclareMathOperator{\esssup}{ess\,sup}
\DeclareMathOperator{\essinf}{ess\,inf}
\DeclareMathOperator{\card}{card}

\begin{document}

\begin{frontmatter}

\title{The Universal Glivenko-Cantelli Property}
\runtitle{The Universal Glivenko-Cantelli Property}
\thankstext{T1}{This work was partially supported by NSF grant 
DMS-1005575.}

\begin{aug}
\author{\fnms{Ramon} \snm{van 
Handel}\ead[label=e2]{rvan@princeton.edu}\protect\thanksref{T1}}
\runauthor{Ramon van Handel}
\affiliation{Princeton University}
\address{Sherrerd Hall, Room 227,\\
Princeton University,\\
Princeton, NJ 08544, USA. \\ \printead{e2}}
\end{aug}

\begin{abstract} 
~\newline\noindent
\begin{minipage}{\getattribute{abstract}{width}}
Let $\mathcal{F}$ be a separable uniformly bounded family of measurable
functions on a standard measurable space $(X,\mathcal{X})$, and let
$N_{[]}(\mathcal{F},\varepsilon,\mu)$ be the smallest number of
$\varepsilon$-brackets in $L^1(\mu)$ needed to cover $\mathcal{F}$.
The following are equivalent:
\begin{enumerate}
\item $\mathcal{F}$ is a universal Glivenko-Cantelli class. 
\item $N_{[]}(\mathcal{F},\varepsilon,\mu)<\infty$ for every
$\varepsilon>0$ and every probability measure $\mu$.
\item $\mathcal{F}$ is totally bounded in $L^1(\mu)$
for every probability measure $\mu$.  
\item $\mathcal{F}$ does not contain a
Boolean $\sigma$-independent sequence.
\end{enumerate}
It follows that universal Glivenko-Cantelli classes are uniformity classes
for general sequences of almost surely convergent random measures.
\end{minipage}
\end{abstract}



\end{frontmatter}

\section{Main results}
\label{sec:intro}

Let $(X,\mathcal{X})$ be a measurable space, and let $\mathcal{F}$ be a 
family of measurable functions on $(X,\mathcal{X})$.  Given a probability 
measure $\mu$ on $(X,\mathcal{X})$, the family $\mathcal{F}$ is 
said to be a \emph{$\mu$-Glivenko-Cantelli class} (cf.\ \cite{Tal87}
or \cite[section 6.6]{Dud99}) if
$$
	\sup_{f\in\mathcal{F}}\left|
	\frac{1}{n}\sum_{k=1}^nf(X_k)-\mu(f)
	\right|\xrightarrow{n\to\infty}0\quad\mbox{a.s.},
$$
where $(X_k)_{k\ge 1}$ is the i.i.d.\ sequence of $X$-valued random 
variables with distribution $\mu$, defined on its canonical product 
probability space.\footnote{
	The supremum in the definition of the $\mu$-Glivenko-Cantelli 
	property need not be measurable in general when the class 
	$\mathcal{F}$ is uncountable.  However, measurability will turn 
	out to hold in the setting of our main results as a consequence 
	of the proofs.  See section \ref{sec:meas} below for 
	further discussion.
} 
The class $\mathcal{F}$ is said to be a \emph{universal Glivenko-Cantelli 
class} if it is $\mu$-Glivenko-Cantelli for every probability measure 
$\mu$ on $(X,\mathcal{X})$.  The goal of this paper is to characterize 
the universal Glivenko-Cantelli property in the case that $\mathcal{F}$ is 
separable and $(X,\mathcal{X})$ is a standard measurable space (these 
regularity assumptions will be detailed below). Somewhat surprisingly, we 
find that universal Glivenko-Cantelli classes are in fact uniformity 
classes for convergence of (random) probability measures under the 
assumptions of this paper, so that their applicability extends 
substantially beyond the setting of laws of large numbers for i.i.d.\ 
sequences that is inherent in their definition.

The following probability-free independence properties for families
of functions will play a fundamental role in this paper.
These notions date back to Marczewski \cite{Mar48} (for sets)
and Rosenthal \cite{Ros74} (for functions, see also \cite{BFT78}).

\begin{defn}
\label{defn:marcz}
A family $\mathcal{F}$ of functions on a set $X$ is said to be
\emph{Boolean independent at levels $(\alpha,\beta)$} if for 
every finite subfamily $\{f_1,\ldots,f_n\}\subseteq\mathcal{F}$
$$
        \bigcap_{j\in F} \{f_j<\alpha\}\cap\bigcap_{j\not\in F} 
	\{f_j>\beta\}
        \ne\varnothing\quad\mbox{for every }F\subseteq\{1,\ldots,n\}.
$$
A sequence $(f_i)_{i\in\mathbb{N}}$ is said to be \emph{Boolean 
$\sigma$-independent at levels $(\alpha,\beta)$} if
$$
        \bigcap_{j\in F} \{f_j<\alpha\}\cap\bigcap_{j\not\in F} 
	\{f_j>\beta\}
        \ne\varnothing\quad\mbox{for every }F\subseteq\mathbb{N}.
$$
A family (sequence) of functions is called Boolean 
($\sigma$-)independent if it is Boolean ($\sigma$-)independent at levels 
$(\alpha,\beta)$ for some $\alpha<\beta$.
\end{defn}

We also recall the well-known notions of bracketing and covering numbers.

\begin{defn}
\label{defn:bracketing}
Let $\mathcal{F}$ be a class of functions on a measurable space
$(X,\mathcal{X})$.   Given $\varepsilon>0$ and a probability measure
$\mu$ on $(X,\mathcal{X})$, a pair of measurable functions $f^+,f^-$
such that $f^-\le f^+$ pointwise and $\mu(f^+-f^-)\le\varepsilon$ defines 
an \emph{$\varepsilon$-bracket in $L^1(\mu)$} $[f^-,f^+]:=
\{f:f^-\le f\le f^+\mbox{ pointwise}\}$.  Denote by 
$N_{[]}(\mathcal{F},\varepsilon,\mu)$ the cardinality of the smallest 
collection of $\varepsilon$-brackets in $L^1(\mu)$ covering
$\mathcal{F}$, and by $N(\mathcal{F},\varepsilon,\mu)$ the cardinality 
of the smallest covering of $\mathcal{F}$ by $\varepsilon$-balls in 
$L^1(\mu)$.
\end{defn}

A measurable space $(X,\mathcal{X})$ is said to be \emph{standard} if it 
is Borel-isomorphic to a Polish space.  A class of functions $\mathcal{F}$ 
on a set $X$ will be said to be \emph{separable} if it contains a 
countable dense subset for the topology of pointwise convergence in 
$\mathbb{R}^X$.\footnote{
This notion of separability is not commonly considered in empirical 
process theory.  A sequential counterpart is more familiar:  $\mathcal{F}$ 
is called pointwise measurable if it contains a countable subset 
$\mathcal{F}_0$ such that every $f\in\mathcal{F}$ is the pointwise limit 
of a sequence in $\mathcal{F}$ (cf.\ \cite[Example 2.3.4]{VW96}).  In 
general, separability is much weaker than pointwise measurability.  
However, a deep result of Bourgain, Fremlin and Talagrand 
\cite[Theorem 4D(viii)$\Rightarrow$(vi)]{BFT78} implies that a separable
uniformly bounded family of measurable functions on a standard space
is necessarily pointwise measurable if it contains no Boolean 
$\sigma$-independent sequence.  Thus universal Glivenko-Cantelli classes
satisfying the assumptions of Theorem \ref{thm:main} below are always 
pointwise measurable, though this is far from obvious a priori.  This fact 
will not be needed in our proofs.
}
We can now formulate our main result.

\begin{thm}
\label{thm:main}
Let $\mathcal{F}$ be a separable uniformly bounded family of measurable
functions on a standard measurable space $(X,\mathcal{X})$.  The 
following are equivalent:
\begin{enumerate}
\item $\mathcal{F}$ is a universal Glivenko-Cantelli class.
\item $N_{[]}(\mathcal{F},\varepsilon,\mu)<\infty$ for every
$\varepsilon>0$ and every probability measure $\mu$.
\item $N(\mathcal{F},\varepsilon,\mu)<\infty$ for every
$\varepsilon>0$ and every probability measure $\mu$.
\item $\mathcal{F}$ contains no Boolean $\sigma$-independent sequence.
\end{enumerate}
\end{thm}

A notable aspect of this result is that the four equivalent conditions 
of Theorem \ref{thm:main} are quite different in nature: roughly speaking, 
the first condition is probabilistic, the second and third are geometric 
and the fourth is combinatorial.

The implication $1\Rightarrow 2$ in Theorem \ref{thm:main} is the most 
important result of this paper.  A consequence of this implication is that 
universal Glivenko-Cantelli classes can be characterized as uniformity 
classes in a much more general setting.

\begin{cor}
\label{cor:uniformity}
Under the assumptions of Theorem \ref{thm:main}, the following are 
equivalent to the equivalent conditions 1--4 of Theorem \ref{thm:main}:
\begin{enumerate}
\setcounter{enumi}{4}
\item For any probability measure $\mu$ on $(X,\mathcal{X})$ and net of 
probability measures
$(\mu_\tau)_{\tau\in I}$ such that $\mu_\tau\to\mu$ setwise, we have
$\sup_{f\in\mathcal{F}}|\mu_\tau(f)-\mu(f)|\to 0$.
\item For any probability measure $\mu$ on $(X,\mathcal{X})$ and sequence 
of random probability measures (kernels) $(\mu_n)_{n\in\mathbb{N}}$ such 
that $\mu_n(A)\to\mu(A)$ a.s.\ for every $A\in\mathcal{X}$, 
we have $\sup_{f\in\mathcal{F}}|\mu_n(f)-\mu(f)|\to 0$ a.s.
\item For any countably generated reverse filtration 
$(\mathcal{G}_{-n})_{n\in\mathbb{N}}$ and $X$-valued random variable $Z$,
$\sup_{f\in\mathcal{F}}
|\mathbf{P}_{\mathcal{G}_{-n}}(f(Z))-\mathbf{P}_{\mathcal{G}_{-\infty}}
(f(Z))|\to 0$ a.s.
\item For any strictly stationary sequence $(Z_n)_{n\in\mathbb{N}}$ of 
$X$-valued random variables, 
$\sup_{f\in\mathcal{F}}|\frac{1}{n}\sum_{k=1}^nf(Z_k)-
\mathbf{P}_{\mathcal{I}}(f(Z_0))|\to 0$ a.s.\ ($\mathcal{I}$ is the 
invariant $\sigma$-field).
\end{enumerate}
Here $\mathbf{P}_{\mathcal{G}}$ denotes any version
of the regular conditional probability 
$\mathbf{P}[\,\cdot\,|\mathcal{G}]$.
\end{cor}

The characterization provided by Theorem \ref{thm:main} and Corollary 
\ref{cor:uniformity} is proved under three regularity assumptions: that 
$\mathcal{F}$ is uniformly bounded and separable, and that 
$(X,\mathcal{X})$ is standard.  It is not difficult to show that any 
universal Glivenko-Cantelli class is uniformly bounded up to additive 
constants (see, for example, \cite[Proposition 4]{DGZ91}), so that the 
assumption that $\mathcal{F}$ is uniformly bounded is not a restriction.  
We will presently argue, however, that without the remaining two 
assumptions a characterization along the lines of this paper cannot be 
expected to hold in general.

In the case that $\mathcal{F}$ is not separable, there are easy 
counterexamples to Theorem \ref{thm:main}. For example, consider the class 
$\mathcal{F}$ consisting of all indicator functions of finite subsets of 
$X$.  It is clear that this class is not $\mu$-Glivenko-Cantelli for any 
nonatomic measure $\mu$, yet condition 3 of Theorem \ref{thm:main} holds. 
Conversely, \cite[section 1.2]{AN10} gives a simple example of a universal 
Glivenko-Cantelli class (in fact, a Vapnik-Chervonenkis class that is 
image admissible Suslin, cf.\ \cite[Corollary 6.1.10]{Dud99}) for which 
condition 8 of Corollary \ref{cor:uniformity}, and therefore condition 2 
of Theorem \ref{thm:main}, are violated.

In the case that $(X,\mathcal{X})$ is not standard, an easy counterexample 
to Theorem \ref{thm:main} is obtained by choosing $X=[0,1]$ and 
$\mathcal{X}=2^X$.  Assuming the continuum hypothesis, nonatomic 
probability measures on $(X,\mathcal{X})$ do not exist \cite[Theorem 
C.1]{Dud02}, so that any uniformly bounded family of functions is 
trivially universal Glivenko-Cantelli.  But we can clearly choose a 
uniformly bounded Boolean $\sigma$-independent sequence $\mathcal{F}$ of 
functions on $X$, in contradiction to Theorem \ref{thm:main}.  This 
example is arguably pathological, but various examples given by Dudley, 
Gin{\'e} and Zinn \cite{DGZ91} show that such phenomena can appear even in 
Polish spaces if we admit universally measurable functions.  Therefore,
in the absence of some regularity assumption on $(X,\mathcal{X})$, the 
universal Glivenko-Cantelli property can be surprisingly broad.  In 
Appendix \ref{sec:counter}, we show that it is consistent with the usual 
axioms of set theory that the implications in Theorem \ref{thm:main} whose 
proof relies on the assumption that $(X,\mathcal{X})$ is standard may fail 
in a general measurable space.  I do not know whether it is possible to 
obtain examples of this type that do not depend on additional 
set-theoretic axioms.

For the case where $(X,\mathcal{X})$ is a general measurable space we will 
prove the following quantitative result, which is of independent interest.

\begin{defn}
\label{defn:shatter}
Let $\gamma>0$.  A family $\mathcal{F}$ of functions on a set $X$ is said 
to \emph{$\gamma$-shatter} a subset $X_0\subseteq X$ if there exist
levels $\alpha<\beta$ with $\beta-\alpha\ge\gamma$ such that, for every 
finite subset $\{x_1,\ldots,x_n\}\subseteq X_0$, the following 
holds:
$$
	\forall\,F\subseteq\{1,\ldots,n\},~~
	\exists\,f\in\mathcal{F}~~
	\mbox{so that}~~
	f(x_j)<\alpha\mbox{ for }j\in F,~~
	f(x_j)>\beta\mbox{ for }j\not\in F.
$$
The \emph{$\gamma$-dimension} of $\mathcal{F}$ is the maximal 
cardinality of $\gamma$-shattered finite subsets of $X$.
\end{defn}

\begin{thm}
\label{thm:scales}
Let $\mathcal{F}$ be a separable uniformly bounded family of 
measurable functions on a measurable space $(X,\mathcal{X})$, 
and let $\gamma>0$.  Consider:
\begin{enumerate}
\renewcommand{\theenumi}{\alph{enumi}}
\item $\mathcal{F}$ has finite $\gamma$-dimension.
\item No sequence in $\mathcal{F}$ is Boolean independent
at levels $(\alpha,\beta)$ with $\beta-\alpha\ge\gamma$.
\item $N_{[]}(\mathcal{F},\varepsilon,\mu)<\infty$ for every
$\varepsilon>\gamma$ and every probability measure $\mu$.
\end{enumerate}
Then the implications $a\Rightarrow b\Rightarrow c$ hold.
\end{thm}

The notion of $\gamma$-dimension appears in Alon et al.\ \cite{ABCH97} 
(called $V_{\gamma/2}$-dimension there). The implication $a\Rightarrow c$ 
of Theorem \ref{thm:scales} contains the recent results of Adams and Nobel 
\cite{AN10a,AN10b,AN10}.  Let us note that condition $b$ is strictly 
weaker than condition $a$: for example, the class 
$\mathcal{F}=\{\mathbf{1}_C:C\mbox{ is a finite subset of }\mathbb{N}\}$ 
has infinite $\gamma$-dimension for $\gamma<1$, but does not contain a 
Boolean independent sequence.  Similarly, condition $c$ is strictly weaker 
than condition $b$: if 
$X=\{x\in\{0,1\}^{\mathbb{N}}:\lim_{n\to\infty}x_n=0\}$ and 
$\mathcal{F}=\{\mathbf{1}_{\{x\in X:x_j=1\}}:j\in\mathbb{N}\}$, then 
$\mathcal{F}$ contains a Boolean independent sequence, but all the 
bracketing numbers are finite as $X$ is countable (note that
$\mathcal{F}$ does not contain a Boolean $\sigma$-independent 
sequence, so there is no contradiction with Theorem \ref{thm:main}).
Condition $b$ is dual 
(in the sense of Assouad \cite{Ass83}) to the nonexistence of a 
$\gamma$-shattered sequence in $X$.  A connection between the latter and 
the universal Glivenko-Cantelli property for families of indicators is 
considered by Dudley, Gin{\'e} and Zinn \cite{DGZ91}.

An interesting question arising from Theorem \ref{thm:scales} is as 
follows.  If $\mathcal{F}$ is uniformly bounded and has finite 
$\gamma$-dimension for all $\gamma>0$, then 
$\sup_{\mu}N(\mathcal{F},\gamma,\mu)<\infty$ for all $\gamma>0$, that is, 
the covering numbers of $\mathcal{F}$ are bounded uniformly with respect 
to the underlying probability measure (see \cite{MV03} for a quantitative 
statement).  If $\mathcal{F}$ is a family of indicators, we have in 
fact the polynomial bound 
$\sup_{\mu}N(\mathcal{F},\varepsilon,\mu)\lesssim \varepsilon^{-d}$
\cite[Theorem 4.6.1]{Dud99}.  In view of Theorem \ref{thm:scales}, one 
might ask whether one can similarly obtain uniform or quantitative bounds 
on the bracketing numbers of $\mathcal{F}$.  Unfortunately, this is not 
the case: $N_{[]}(\mathcal{F},\varepsilon,\mu)$ can blow up arbitrarily 
quickly as $\varepsilon\downarrow 0$.
The following result is based on a combinatorial 
construction of Alon, Haussler, and Welzl \cite{AHW87}.

\begin{prop}
\label{prop:vc}
There exists a countable class $\mathcal{C}$ of subsets of $\mathbb{N}$, 
whose Vapnik-Chervonenkis dimension is two (that is, the 
$\gamma$-dimension of $\{\mathbf{1}_C:C\in\mathcal{C}\}$ is two for all 
$0<\gamma<1$) such that the following holds:
for any function $n(\varepsilon)\uparrow\infty$ as 
$\varepsilon\downarrow 0$, there is a probability measure $\mu$ on 
$\mathbb{N}$ such that $N_{[]}(\mathcal{C},\varepsilon,\mu)\ge 
n(\varepsilon)$ for all $0<\varepsilon<1/3$.
In particular, $\sup_{\mu}N_{[]}(\mathcal{C},\varepsilon,\mu)=\infty$ for 
all $0<\varepsilon<1/3$.
\end{prop}

Probabilistically, this result has the following consequence.  In contrast 
to the universal Glivenko-Cantelli property, it is known that both the 
uniform Glivenko-Cantelli property and the universal Donsker property are 
equivalent to finiteness of the Vapnik-Chervonenkis dimension for image 
admissible Suslin classes of sets (see \cite{Dud99}, p.\ 225 and p.\ 215, 
respectively). These results are proved using symmetrization arguments. In 
view of Theorem \ref{thm:scales}, one might expect that it is possible to 
provide an alternative proof of these results for separable classes using 
bracketing methods (as in \cite[Chapter 7]{Dud99}). However, this would 
require either uniform or quantitative control of the bracketing numbers, 
both of which are ruled out by Proposition \ref{prop:vc}.

The original motivation of the author was an attempt to characterize 
uniformity classes for reverse martingales that appear in filtering 
theory.  In a recent paper, Adams and Nobel \cite{AN10} showed that 
Vapnik-Chervonenkis classes of sets are uniformity classes for the 
convergence of empirical measures of stationary ergodic sequences; their 
proof could be extended to more general random measures.  A simplified 
argument, which makes the connection with bracketing, appeared 
subsequently in \cite{AN10b}. While attempting to understand the results 
of \cite{AN10}, the author realized that the techniques used in the proof 
are closely related to a set of techniques developed by Bourgain, Fremlin 
and Talagrand \cite{BFT78,Tal84} to study pointwise compact sets of 
measurable functions.  The proof of Theorem \ref{thm:main} is based on 
this elegant theory, which does not appear to be well known in the 
probability literature (however, the proofs of our main results, Theorem 
\ref{thm:main}, Corollary \ref{cor:uniformity}, and Theorem 
\ref{thm:scales}, are intended to be essentially self-contained).

A key innovation in this paper is the construction in section 
\ref{sec:scales} of a ``weakly dense'' set which allows to prove the 
implication $4\Rightarrow 2$ in Theorem \ref{thm:main} (and $b\Rightarrow 
c$ in Theorem \ref{thm:scales}).  This result is the essential step that 
closes the circle of implications in Theorem \ref{thm:main} and Corollary 
\ref{cor:uniformity}. Many of the remaining implications are essentially 
known, albeit in more restrictive settings and/or using significantly more 
complicated proofs: these results are unified here in what appears to be 
(in view the simplicity of the proofs and the counterexamples above and in 
Appendix \ref{sec:counter}) their natural setting.  In a topological 
setting (continuous functions on a compact space), the equivalence of 
$1,3,4$ in Theorem \ref{thm:main} can be deduced by combining 
\cite[Theorem 14-1-7]{Tal84} with Talagrand's characterization of the 
$\mu$-Glivenko-Cantelli property \cite[Theorem 11-1-1]{Tal84}, 
\cite{Tal87} (note that in this setting the distinction between Boolean 
independent and $\sigma$-independent sequences is irrelevant).  The 
equivalence between $3,4$ in Theorem \ref{thm:main} is also obtained in 
\cite[Theorem 4D]{BFT78} by a much more complicated method.  The 
implication $5\Rightarrow 2$ follows from the characterization of 
uniformity classes for setwise convergence of Stute \cite{Stu76} and 
Tops{\o}e \cite{Top77}.  The implications $2\Rightarrow 1,5$--$8$ follow 
from the classical Blum-DeHardt argument, up to measurability problems 
that are resolved here.  Finally, the implication $a\Rightarrow c$ (but 
not $b\Rightarrow c$) of Theorem \ref{thm:scales} is shown in \cite{AN10b} 
for the special case of Vapnik-Chervonenkis classes of sets.  

The remainder of this paper is organized as follows.  We first prove 
Theorem \ref{thm:scales} in section \ref{sec:scales}.  The proofs of 
Theorem \ref{thm:main}, Corollary \ref{cor:uniformity}, and Proposition 
\ref{prop:vc} are subsequently given in sections \ref{sec:main}, 
\ref{sec:uniformity}, and \ref{sec:vc}, respectively.  Finally, Appendix 
\ref{app:boole} and Appendix \ref{app:decomp} develop some properties of 
Boolean $\sigma$-independent sequences and decomposition theorems that are 
used in the proofs of our main results, while Appendix \ref{sec:counter} 
is devoted to the aforementioned counterexamples to Theorem 
\ref{thm:main} in nonstandard spaces.

\section{Proof of Theorem \ref{thm:scales}}
\label{sec:scales}

In this section, we fix a measurable space $(X,\mathcal{X})$ and a 
separable uniformly bounded family of measurable functions $\mathcal{F}$.  
Let $\mathcal{F}_0\subseteq\mathcal{F}$ be a countable family that is 
dense in $\mathcal{F}$ in the pointwise convergence topology.

\begin{defn}
Denote by $\Pi(X,\mathcal{X})$ the collection of all finite measurable 
partitions of $X$.  For $\pi,\pi'\in\Pi(X,\mathcal{X})$, we write 
$\pi\preceq\pi'$ if $\pi$ is finer than $\pi'$.  For any pair of sets
$A,B\in\mathcal{X}$, finite partition $\pi\in\Pi(X,\mathcal{X})$, and 
probability measure $\mu$ on $(X,\mathcal{X})$, define the $\mu$-essential
$\pi$-boundary of $(A,B)$ as
$$
	\partial_\pi^\mu(A,B) = \bigcup\{P\in\pi:
	\mu(P\cap A)>0\mbox{ and }\mu(P\cap B)>0\}.
$$
\end{defn}

We begin by proving an approximation result.

\begin{lem}
\label{lem:approx}
Let $\mu$ be a probability measure on $(X,\mathcal{X})$ and
let $\gamma>0$.  If
$$
	\inf_{\pi\in\Pi(X,\mathcal{X})}\sup_{f\in\mathcal{F}_0}
	\mu\big(\partial_\pi^\mu(\{f<\alpha\},\{f>\beta\})\big)=0
	\quad\mbox{for all}\quad\beta-\alpha\ge\gamma,
$$
then $N_{[]}(\mathcal{F},\varepsilon,\mu)<\infty$ for every
$\varepsilon>\gamma$.
\end{lem}

\begin{proof}
There is clearly no loss of generality in assuming that every
$f\in\mathcal{F}$ takes values in $\mbox{}[0,1]\mbox{}$
and that $\gamma<1$.  
Fix $k\ge 1$, and let $\delta:=\gamma/k$.  Choose 
$\pi\in\Pi(X,\mathcal{X})$ so that 
$$
	\sup_{f\in\mathcal{F}_0}
	\mu\left(\Xi(f)\right)
	<\delta,\qquad
	\Xi(f):=
	\bigcup_{1\le j\le \lfloor\delta^{-1}\rfloor}
	\partial_\pi^\mu(\{f<j\delta\},
	\{f>j\delta+\gamma\}).
$$
For each $f\in\mathcal{F}_0$, define the functions $f^+$ and $f^-$ as 
follows:
\begin{align*}
	f^+ &= \delta\,\lceil\delta^{-1}\rceil\,
	\mathbf{1}_{\Xi(f)} + \sum_{P\in\pi:P\not\subseteq\Xi(f)}
	\delta\,\lceil 
	\delta^{-1}\esssup_{P}f\rceil\,\mathbf{1}_P,
	\\
	f^- &= \sum_{P\in\pi:P\not\subseteq\Xi(f)}
	\delta\,\lfloor 
	\delta^{-1}\essinf_{P}f\rfloor\,\mathbf{1}_P.
\end{align*}
Here $\esssup_Pf$ ($\essinf_Pf$) denotes the essential supremum (infimum) 
of $f$ on the set $P$ with respect to $\mu$.  By construction,
$f^-\le f\le f^+$ outside a $\mu$-null set and $\mu(f^+-f^-) < 
\gamma+3\delta$.  Moreover, as $f^+,f^-$ are
constant on each $P\in\pi$ and take values in the finite set 
$\{j\delta:0\le j\le \lceil\delta^{-1}\rceil\}$, there is only a finite 
number of such functions.  As $\mathcal{F}_0$ is countable, we can 
eliminate the null set to obtain a finite number of 
$(\gamma+3\delta)$-brackets in $L^1(\mu)$ covering $\mathcal{F}_0$.  But 
$\mathcal{F}_0$ is pointwise dense in $\mathcal{F}$, so 
$N_{[]}(\mathcal{F},\gamma+3\delta,\mu)<\infty$, and we may choose 
$\delta=\gamma/k$ arbitrarily small.
\end{proof}

To proceed, we need the notion of a ``weakly dense'' set, which is the 
measure-theoretic counterpart of the corresponding topological notion 
defined in \cite{BFT78}.

\begin{defn}
\label{defn:wkdens}
Given a measurable set $A\in\mathcal{X}$ and a probability measure
$\mu$ on $(X,\mathcal{X})$, the family of functions $\mathcal{F}$ is said 
to be \emph{$\mu$-weakly dense over $A$ at levels $(\alpha,\beta)$} if
$\mu(A)>0$ and for any finite collection of measurable sets 
$B_1,\ldots,B_p\in\mathcal{X}$ such that $\mu(A\cap B_i)>0$ for
all $1\le i\le p$, there exists $f\in\mathcal{F}$ such that
$\mu(A\cap B_i\cap\{f<\alpha\})>0$ and $\mu(A\cap B_i\cap\{f>\beta\})>0$
for all $1\le i\le p$.
\end{defn}

The key idea of this section, which lies at the heart of the results in 
this paper, is that we can construct such a set if the bracketing numbers 
fail to be finite.  The proof is straightforward but requires some 
elementary topological notions: the reader unfamiliar with nets is 
referred to the classic text \cite{Kel55}, while weak compactness of the 
unit ball in $L^2$ follows from Alaoglu's theorem \cite[Theorem 
V.3.1]{Con85}.

\begin{prop}
\label{prop:wkdens}
Suppose there exists a probability measure $\mu$ on $(X,\mathcal{X})$
such that $N_{[]}(\mathcal{F},\varepsilon,\mu)=\infty$ for some
$\varepsilon>\gamma$.  Then there exist $\alpha<\beta$ with
$\beta-\alpha\ge\gamma$ and a measurable set $A\in\mathcal{X}$ such that
$\mathcal{F}_0$ is $\mu$-weakly dense over $A$ at levels $(\alpha,\beta)$.
\end{prop}

\begin{proof}
By Lemma \ref{lem:approx}, there exist $\alpha<\beta$ with
$\beta-\alpha\ge\gamma$ such that
$$
	\inf_{\pi\in\Pi(X,\mathcal{X})}\sup_{f\in\mathcal{F}_0}
	\mu\big(\partial_\pi^\mu(\{f<\alpha\},\{f>\beta\})\big)>0.
$$
Choose for every $\pi\in\Pi(X,\mathcal{X})$ a function
$f_\pi\in\mathcal{F}_0$ such that 
$$
	\mu\big(\partial_\pi^\mu(\{f_\pi<\alpha\},\{f_\pi>\beta\})\big)
	\ge \frac{1}{2}
	\sup_{f\in\mathcal{F}_0}
	\mu\big(\partial_\pi^\mu(\{f<\alpha\},\{f>\beta\})\big).
$$
Define $A_\pi:=\partial_\pi^\mu(\{f_\pi<\alpha\},\{f_\pi>\beta\})$.  Then
$(\mathbf{1}_{A_\pi})_{\pi\in\Pi(X,\mathcal{X})}$ is a net of random 
variables in the unit ball of $L^2(\mu)$.  By weak compactness, there is 
for some directed set $T$ a subnet 
$(\mathbf{1}_{A_{\pi(\tau)}})_{\tau\in T}$ that
converges weakly in $L^2(\mu)$ to a random variable $H$.  We claim that 
$\mathcal{F}_0$ is $\mu$-weakly dense over $A:=\{H>0\}$ at levels 
$(\alpha,\beta)$.

To prove the claim, let us first note that as $\inf_\pi\mu(A_\pi)>0$, 
clearly $\mu(A)>0$.  Now fix $B_1,\ldots,B_p\in\mathcal{X}$ such that
$\mu(A\cap B_i)>0$ for all $i$.  This trivially implies 
that $\mu(H\mathbf{1}_{A\cap B_i})>0$ for all $i$, so we can
choose $\tau_0\in T$ such that
$$
	\mu(A_{\pi(\tau)}\cap A\cap B_i)>0\quad
	\forall\,1\le i\le p,~\tau\preceq\tau_0.
$$
Let $\pi_0$ be the partition generated by $A,B_1,\ldots,B_p$, and choose 
$\tau^*\in T$ such that $\tau^*\preceq\tau_0$ and 
$\pi^*:=\pi(\tau^*)\preceq\pi_0$.  As $A\cap B_i$ is a union of atoms of 
$\pi^*$ by construction, $\mu(A_{\pi^*}\cap A\cap B_i)>0$ must imply
that $A\cap B_i$ contains an atom $P\in\pi^*$ such that 
$\mu(P\cap\{f_{\pi^*}<\alpha\})>0$ and
$\mu(P\cap\{f_{\pi^*}>\beta\})>0$.  Therefore
$$
	\mu(A\cap B_i\cap\{f_{\pi^*}<\alpha\})>0\quad
	\mbox{and}\quad
	\mu(A\cap B_i\cap\{f_{\pi^*}>\beta\})>0
	\quad\forall\,i.
$$
Thus $\mathcal{F}_0$ is $\mu$-weakly dense over $A$ at levels
$(\alpha,\beta)$ as claimed.
\end{proof}

We can now complete the proof of Theorem \ref{thm:scales}.

\begin{proof}[Theorem \ref{thm:scales}]
~

$a\Rightarrow b$: Lemma \ref{lem:assouad} in Appendix \ref{app:boole}
shows that if $\mathcal{F}$ contains a subset of cardinality $2^n$ 
that is Boolean independent at levels $(\alpha,\beta)$ with 
$\beta-\alpha\ge\gamma$, then $\mathcal{F}$ $\gamma$-shatters a subset of 
$X$ of cardinality $n$.  Therefore, if condition $b$ fails, there exist 
$\gamma$-shattered finite subsets of $X$ of arbitrarily large cardinality, 
in contradiction with condition $a$.

$b\Rightarrow c$: Suppose that condition $c$ fails.  By Proposition 
\ref{prop:wkdens}, there exist a probability measure $\mu$, levels 
$\alpha<\beta$ with $\beta-\alpha\ge\gamma$, and a set 
$A\in\mathcal{X}$ so that $\mathcal{F}_0$ is $\mu$-weakly dense 
over $A$ at levels $(\alpha,\beta)$.  We now iteratively apply 
Definition \ref{defn:wkdens} to construct a Boolean independent sequence.  
Indeed, applying first the definition with
$p=1$ and $B_1=X$, we choose $f_1\in\mathcal{F}_0$ so that
$\mu(A\cap\{f_1<\alpha\})>0$ and $\mu(A\cap\{f_1>\beta\})>0$.
Then applying the definition with $p=2$ and $B_1=\{f_1<\alpha\}$,
$B_2=\{f_1>\beta\}$, we choose $f_2\in\mathcal{F}_0$ so that
$\mu(A\cap\{f_1<\alpha\}\cap\{f_2<\alpha\})>0$,
$\mu(A\cap\{f_1<\alpha\}\cap\{f_2>\beta\})>0$,
$\mu(A\cap\{f_1>\beta\}\cap\{f_2<\alpha\})>0$, and
$\mu(A\cap\{f_1>\beta\}\cap\{f_2>\beta\})>0$.  Repeating this procedure
yields the desired sequence $(f_i)_{i\in\mathbb{N}}$.
\end{proof}

\section{Proof of Theorem \ref{thm:main}}
\label{sec:main}

Throughout this section, we fix a standard measurable space 
$(X,\mathcal{X})$ and a separable uniformly bounded family of measurable 
functions $\mathcal{F}$.  We will prove Theorem \ref{thm:main} by
proving the implications $1\Rightarrow 4\Rightarrow 2\Rightarrow 1$
and $2\Rightarrow 3\Rightarrow 4$.  

\subsection{$1\Rightarrow 4$}

Suppose there exists a sequence 
$(f_i)_{i\in\mathbb{N}}\subseteq\mathcal{F}$ that is Boolean 
$\sigma$-independent at levels $(\alpha,\beta)$ for some $\alpha<\beta$.  
Clearly we must have 
$$
	\kappa_-<\alpha<\beta<\kappa_+,\qquad
	\kappa_-:= \inf_{f\in\mathcal{F}}\inf_{x\in X}f(x),\quad
	\kappa_+:= \sup_{f\in\mathcal{F}}\sup_{x\in X}f(x).
$$
Let $p=(\kappa_+-\beta+\varepsilon)/(\kappa_+-\alpha)$, where
we choose $\varepsilon>0$ such that $p<1$.  Applying Theorem 
\ref{thm:marczewski} in Appendix \ref{app:boole} to the sets
$A_i=\{f_i<\alpha\}$ and $B_i=\{f_i>\beta\}$, there exists a 
probability measure $\mu$ on $(X,\mathcal{X})$ such that
$(\{f_i<\alpha\})_{i\in\mathbb{N}}$ is an i.i.d.\ sequence of
sets with $\mu(\{f_i<\alpha\})=\mu(X\backslash\{f_i>\beta\})=p$ for every
$i\in\mathbb{N}$.

We now claim that $\mathcal{F}$ is not $\mu$-Glivenko-Cantelli, which 
yields the desired contradiction.  To this end, note that we can trivially 
estimate for any $f\in\mathcal{F}$
$$
	\beta\,\mathbf{1}_{f>\beta}+\kappa_-\,\mathbf{1}_{f\le\beta}
	\le
	f \le \alpha\,\mathbf{1}_{f<\alpha}+\kappa_+\,\mathbf{1}_{f\ge\alpha}.
$$
We therefore have
\begin{align*}
	\sup_{f\in\mathcal{F}}\left|
	\frac{1}{n}\sum_{k=1}^nf(X_k)-\mu(f)
	\right| &\mbox{}\ge
	\sup_{j\in\mathbb{N}}
	\frac{1}{n}\sum_{k=1}^n\{f_j(X_k)-\mu(f_j)\} \\
	&\mbox{}\ge
	(\kappa_--\beta)
	\inf_{j\in\mathbb{N}}
	\frac{1}{n}\sum_{k=1}^n
	\mathbf{1}_{f_j\le\beta}(X_k)
	+\varepsilon.
\end{align*}
But if $(X_k)_{k\ge 1}$ are i.i.d.\ with distribution $\mu$ then, by 
construction, the family of random variables 
$\{\mathbf{1}_{f_j\le\beta}(X_k):j,k\in\mathbb{N}\}$ is i.i.d.\ with 
$\mathbf{P}[\mathbf{1}_{f_j\le\beta}(X_k)=0]>0$, so
$$
	\inf_{j\in\mathbb{N}}
	\frac{1}{n}\sum_{k=1}^n
	\mathbf{1}_{f_j\le\beta}(X_k) = 0 \quad\mbox{a.s.}\quad
	\mbox{for all }n\in\mathbb{N}.	
$$
Thus $\mathcal{F}$ is not a $\mu$-Glivenko-Cantelli class.  This 
completes the proof.

\subsection{$4\Rightarrow 2$}

Suppose there exists a probability measure $\mu$ and $\varepsilon>0$ such 
that $N_{[]}(\mathcal{F},\varepsilon,\mu)=\infty$.  By Proposition 
\ref{prop:wkdens}, there exist levels 
$\alpha<\beta$ and a set $A\in\mathcal{X}$ such that 
$\mathcal{F}$ is $\mu$-weakly dense over $A$ at levels $(\alpha,\beta)$.
We will presently construct a Boolean $\sigma$-independent sequence, 
which yields the desired contradiction.  The idea is to repeat the proof 
of Theorem \ref{thm:scales}, but now exploiting the fact that 
$(X,\mathcal{X})$ is standard to ensure that the 
infinite intersections in the definition of Boolean $\sigma$-independence
are nonempty.

As $(X,\mathcal{X})$ is standard, we may assume without loss of 
generality that $X$ is Polish and that $\mathcal{X}$ is the Borel
$\sigma$-field.  Thus $\mu$ is inner regular.  We now apply Definition 
\ref{defn:wkdens} as follows.  First, setting $p=1$ and $B_1=X$, choose 
$f_1\in\mathcal{F}$ such that 
$$
	\mu(A\cap\{f_1<\alpha\})>0,\qquad\mu(A\cap\{f_1>\beta\})>0.
$$
As $\mu$ is inner regular, we may choose compact sets 
$F_1\subseteq\{f_1<\alpha\}$ and $G_1\subseteq\{f_1>\beta\}$ such that 
$\mu(A\cap F_1)>0$ and $\mu(A\cap F_2)>0$.  Applying the definition with 
$p=2$, $B_1=F_1$, and $B_2=G_1$, we can choose $f_2\in\mathcal{F}$ such 
that 
\begin{align*}
	&\mu(A\cap F_1\cap\{f_2<\alpha\})>0,
	&\mu(A\cap F_1\cap\{f_2>\beta\})>0, \\
 	&\mu(A\cap G_1\cap\{f_2<\alpha\})>0,
	&\mu(A\cap G_1\cap\{f_2>\beta\})>0.
\end{align*}  
Using again inner regularity, we can now choose compact sets 
$F_2\subseteq\{f_2<\alpha\}$ and $G_2\subseteq\{f_2>\beta\}$ such that 
$\mu(A\cap F_1\cap F_2)>0$, $\mu(A\cap F_1\cap G_2)>0$, $\mu(A\cap G_1\cap 
F_2)>0$, and $\mu(A\cap G_1\cap G_2)>0$.  Iterating the above steps, we 
construct a sequence of functions  
$(f_i)_{i\in\mathbb{N}}\subseteq\mathcal{F}$
and compact sets $(F_i)_{i\in\mathbb{N}}$, $(G_i)_{i\in\mathbb{N}}$ such 
that $F_i\subseteq\{f_i<\alpha\}$, $G_i\subseteq\{f_i>\beta\}$ for
every $i\in\mathbb{N}$, and for any $n\in\mathbb{N}$
$$
	\mu\left(\bigcap_{j\in Q}F_j\cap\bigcap_{j\in
	\{1,\ldots,n\}\backslash Q}G_j
	\right)>0\quad\mbox{for every }Q\subseteq\{1,\ldots,n\}.
$$
Now suppose that the sequence $(f_i)_{i\in\mathbb{N}}$ is not Boolean
$\sigma$-independent.  Then
$$
        \bigcap_{j\in R} \{f_j<\alpha\}\cap\bigcap_{j\not\in R} 
	\{f_j>\beta\} = \varnothing	
$$
for some $R\subseteq\mathbb{N}$.  Thus we certainly have
$$
        \bigcap_{j\in R} F_j\cap\bigcap_{j\not\in R} 
	G_j = \varnothing.
$$
Choose arbitrary $\ell\in R$ (if $R$ is the empty set, replace $F_\ell$ by 
$G_1$ throughout the following argument).  Then clearly $\{X\backslash 
F_j:j\in R\}\cup\{X\backslash G_j:j\not\in R\}$ is an open cover of 
$F_\ell$.  Therefore, there exist finite subsets $Q_1\subseteq R$,
$Q_2\subseteq\mathbb{N}\backslash R$ such that $\{X\backslash    
F_j:j\in Q_1\}\cup\{X\backslash G_j:j\in Q_2\}$ covers $F_\ell$.
But then
$$
        F_\ell\cap\bigcap_{j\in Q_1} F_j\cap\bigcap_{j\in Q_2} 
	G_j = \varnothing,
$$
a contradiction.  Thus $(f_i)_{i\in\mathbb{N}}$ is Boolean 
$\sigma$-independent at levels $(\alpha,\beta)$.

\subsection{$2\Rightarrow 1$}
\label{sec:blum}

This is the usual Blum-DeHardt argument, included here for 
completeness.  Fix a probability measure $\mu$ and $\varepsilon>0$,
and suppose that $N_{[]}(\mathcal{F},\varepsilon,\mu)<\infty$.
Choose $\varepsilon$-brackets $[f_1,g_1],\ldots,[f_N,g_N]$ in $L^1(\mu)$ 
covering $\mathcal{F}$.  Then
\begin{multline*}
	\sup_{f\in\mathcal{F}}|\mu_n(f)-\mu(f)| =
	\sup_{f\in\mathcal{F}}\{\mu_n(f)-\mu(f)\} \vee
	\sup_{f\in\mathcal{F}}\{\mu(f)-\mu_n(f)\}  \\
	\le
	\max_{i=1,\ldots,N}\{\mu_n(g_i)-\mu(f_i)\} \vee
	\max_{i=1,\ldots,N}\{\mu(g_i)-\mu_n(f_i)\} ,
\end{multline*}
where we define the empirical measure $\mu_n:=\frac{1}{n}\sum_{k=1}^n
\delta_{X_k}$ for an i.i.d.\ sequence $(X_k)_{k\in\mathbb{N}}$ with 
distribution $\mu$.  The right hand side in the above expression is 
measurable and converges a.s.\ to a constant not exceeding $\varepsilon$ 
by the law of large numbers.  As $\varepsilon>0$ and $\mu$ were arbitrary,
$\mathcal{F}$ is universal Glivenko-Cantelli.

\subsection{$2\Rightarrow 3\Rightarrow 4$}

As $N(\mathcal{F},\varepsilon,\mu)\le 
N_{[]}(\mathcal{F},2\varepsilon,\mu)$, the implication $2\Rightarrow 3$ is 
trivial.  It therefore remains to prove the implication $3\Rightarrow 4$.

To this end, suppose that there exists a sequence 
$(f_i)_{i\in\mathbb{N}}\subseteq\mathcal{F}$ that is Boolean 
$\sigma$-independent at levels $(\alpha,\beta)$ for some $\alpha<\beta$.
Construct the probability measure $\mu$ as in the proof of the implication
$1\Rightarrow 4$.  We claim that $N(\mathcal{F},\varepsilon,\mu)=\infty$
for $\varepsilon>0$ sufficiently small, which yields the desired 
contradiction.

To prove the claim, it suffices to note that for any $i\ne j$
\begin{align*}
	\mu(|f_i-f_j|) &\ge
	\mu(|f_i-f_j|\mathbf{1}_{f_j<\alpha}\mathbf{1}_{f_i>\beta}) \\
	&\ge (\beta-\alpha)\,
	\mu(\{f_j<\alpha\}\cap\{f_i>\beta\}) =
	(\beta-\alpha)p(1-p)>0
\end{align*}
by the construction of $\mu$.  Therefore $\mathcal{F}$ contains an 
infinite set of $(\beta-\alpha)p(1-p)$-separated points in $L^1(\mu)$,
so $N(\mathcal{F},(\beta-\alpha)p(1-p)/2,\mu)=\infty$.

\subsection{A remark about a.s.\ convergence and measurability}
\label{sec:meas}

When the class $\mathcal{F}$ is only assumed to be separable, the quantity
$$
	\Gamma_n(\mathcal{F},\mu):=
	\sup_{f\in\mathcal{F}}\left|\frac{1}{n}
	\sum_{k=1}^nf(X_k)-\mu(f)\right|
$$
may well be nonmeasurable.  For nonmeasurable functions, there are 
inequivalent notions of convergence that coincide with a.s.\ convergence 
in the measurable case.  In this paper, following Talagrand \cite{Tal87}, 
we defined $\mu$-Glivenko-Cantelli classes as those for which the quantity 
$\Gamma_n(\mathcal{F},\mu)$ converges to zero a.s., that is, pointwise 
outside a set of probability zero.  A different definition, given by 
Dudley \cite[section 3.3]{Dud99}, is to require that
$\Gamma_n(\mathcal{F},\mu)$ converges to zero almost uniformly, that is, 
it is dominated by a sequence of measurable random variables converging to 
zero a.s.

For nonmeasurable functions, almost uniform convergence is in general much 
stronger than a.s.\ convergence.  Nonetheless, in the fundamental paper 
characterizing the $\mu$-Glivenko-Cantelli property, Talagrand showed
\cite[Theorem 22]{Tal87} that for $\mu$-Glivenko-Cantelli classes a.s.\ 
convergence already implies almost uniform convergence.  Thus this is 
certainly the case for universal Glivenko-Cantelli classes.  In the 
setting of Theorem \ref{thm:main}, the latter can also be seen directly:
indeed, the proof of the implication $1\Rightarrow 4$ requires only a.s.\ 
convergence, while the Blum-DeHardt argument $2\Rightarrow 1$ 
automatically yields the stronger notion of almost uniform convergence.

However, let us note that in Corollary \ref{cor:countdens} below we will 
prove an even stronger property: for separable uniformly bounded classes 
$\mathcal{F}$ with finite bracketing numbers, the quantity 
$\sup_{f\in\mathcal{F}}|\nu(f)-\rho(f)|$ is Borel-measurable for arbitrary 
random probability measures $\nu,\rho$. Thus $\Gamma_n(\mathcal{F},\mu)$ 
is automatically measurable for universal Glivenko-Cantelli classes 
satisfying the assumptions of Theorem \ref{thm:main}, though this is far 
from obvious a priori.  Similarly, if any of the equivalent conditions of 
Theorem \ref{thm:main} or Corollary \ref{cor:uniformity} holds, then all 
the suprema in Corollary \ref{cor:uniformity} are measurable. It follows 
that a.s.\ and almost uniform convergence coincide trivially in our main 
results.

\section{Proof of Corollary \ref{cor:uniformity}}
\label{sec:uniformity}

Throughout this section, we fix a standard measurable space 
$(X,\mathcal{X})$ and a separable uniformly bounded family of measurable 
functions $\mathcal{F}$.  We will prove Corollary \ref{cor:uniformity} by 
proving the implications $2\Leftrightarrow 5$ and $2\Rightarrow 
\{6,7,8\}\Rightarrow 1$.  The implication $5\Rightarrow 2$ is related to a 
result of Tops{\o}e \cite{Top77}, though we give here a direct proof 
inspired by Stute \cite{Stu76}.  The remaining implications are 
straightforward modulo measurability issues.

\subsection{$2\Leftrightarrow 5$}

The implication $2\Rightarrow 5$ follows from the Blum-DeHardt argument as 
in section \ref{sec:blum}.  Conversely,
suppose that condition 2 does not hold, so that
$N_{[]}(\mathcal{F},\varepsilon,\mu)=\infty$
for some $\varepsilon>0$ and probability measure $\mu$.  Then by
Lemma \ref{lem:approx}, there exist $\delta>0$ and $\alpha<\beta$ 
such that we can choose for every $\pi\in\Pi(X,\mathcal{X})$ a function
$f_\pi\in\mathcal{F}$ with
$$
	\mu(D_\pi)\ge\delta,\qquad\quad
	D_\pi:=\partial_\pi^\mu(\{f_\pi<\alpha\},\{f_\pi>\beta\}).
$$
We now define for every $\pi\in\Pi(X,\mathcal{X})$ two probability 
measures $\mu_\pi^+,\mu_\pi^-$ as follows.  For every $P\in\pi$ such that
$P\subseteq D_\pi$, choose two points $x_P^+\in P\cap\{f_\pi>\beta\}$ 
and $x_P^-\in P\cap\{f_\pi<\alpha\}$ arbitrarily, and define for 
every $A\in\mathcal{X}$
$$
	\mu_\pi^\pm(A) = 
	\mu(A\backslash D_\pi) +
	\sum_{P\in\pi:P\subseteq D_\pi}
	\mu(P)\,\mathbf{1}_A(x_P^\pm).
$$
Then $(\mu_\pi^\pm)_{\pi\in\Pi(X,\mathcal{X})}$ is a net of
probability measures that converges to $\mu$ setwise: indeed,
for every $A\in\mathcal{X}$, we have
$\mu_\pi^\pm(A)=\mu(A)$ whenever $\pi\preceq\pi_A$ with
$\pi_A=\{A,X\backslash A\}$.  On the other hand, by construction we
have 
$$
	\sup_{f\in\mathcal{F}}|\mu_\pi^+(f)-\mu_\pi^-(f)|
	\ge |\mu_\pi^+(f_\pi)-\mu_\pi^-(f_\pi)| \ge
	(\beta-\alpha)\mu(D_\pi)\ge(\beta-\alpha)\delta
$$
for every $\pi\in\Pi(X,\mathcal{X})$.  Therefore either
$(\mu_\pi^+)_{\pi\in\Pi(X,\mathcal{X})}$ or
$(\mu_\pi^-)_{\pi\in\Pi(X,\mathcal{X})}$ does not converge 
to $\mu$ uniformly over $\mathcal{F}$, in contradiction to condition 5.

\subsection{$2\Rightarrow\{6,7,8\}$}

The implication $2\Rightarrow 6$ follows immediately from the Blum-DeHardt 
argument as in section \ref{sec:blum}.  The complication for the 
implications $2\Rightarrow\{7,8\}$ is that the limiting measure is a 
random measure (unlike $2\Rightarrow 6$ where the limiting measure is 
nonrandom).  Intuitively one can simply condition on 
$\mathcal{G}_{-\infty}$ or $\mathcal{I}$, respectively, so that the 
problem reduces to the implication $2\Rightarrow 6$ under the conditional 
measure.  The main work in the proof consists of resolving the 
measurability issues that arise in this approach.

Let $\mathcal{F}_0\subseteq\mathcal{F}$ be a countable family that is 
dense in $\mathcal{F}$ in the topology of pointwise convergence.  
We first show that $\mathcal{F}_0$ is also $L^1(\mu)$-dense in 
$\mathcal{F}$ for any $\mu$: this is not obvious, as the dominated 
convergence theorem does not hold for nets.

\begin{lem}
\label{lem:ptmsc}
If $N_{[]}(\mathcal{F},\varepsilon,\mu)<\infty$ for all
$\varepsilon>0$, then $\mathcal{F}_0$ is $L^1(\mu)$-dense in 
$\mathcal{F}$.
\end{lem}

\begin{proof}
Fix $\varepsilon>0$, and choose 
$\varepsilon$-brackets $[f_1,g_1],\ldots,[f_N,g_N]$ in $L^1(\mu)$
covering $\mathcal{F}$.  As topological closure and finite unions commute,
for every $f\in\mathcal{F}$ there exists $1\le i\le N$ such that $f$
is in the pointwise closure of $[f_i,g_i]\cap\mathcal{F}_0$.  But then
clearly $f\in[f_i,g_i]$, and choosing any $g\in 
[f_i,g_i]\cap\mathcal{F}_0$ we have 
$\mu(|f-g|)\le\mu(g_i-f_i)\le\varepsilon$.  As $\varepsilon>0$ is 
arbitrary, the proof is complete.
\end{proof}

We can now reduce the suprema in conditions $7$ and $8$ to countable 
suprema. 

\begin{cor}
\label{cor:countdens}
Suppose that $N_{[]}(\mathcal{F},\varepsilon,\mu)<\infty$ for every
$\varepsilon>0$ and probability measure $\mu$.  Then for any pair
of probability measures $\mu,\nu$ we have
$$
	\sup_{f\in\mathcal{F}}|\mu(f)-\nu(f)| =
	\sup_{f\in\mathcal{F}_0}|\mu(f)-\nu(f)|.
$$
In particular, this holds when $\mu$ and $\nu$ are random measures.
\end{cor}

\begin{proof}
Fix (nonrandom) probability measures $\mu,\nu$, and define
$\rho=\{\mu+\nu\}/2$.  Then $\mathcal{F}_0$ is $L^1(\rho)$-dense
in $\mathcal{F}$ by Lemma \ref{lem:ptmsc}.  In particular, for every
$f\in\mathcal{F}$ and $\varepsilon>0$, we can choose $g\in\mathcal{F}_0$ 
such that $\mu(|f-g|)+\nu(|f-g|)\le \varepsilon$.  Now let 
$(f_n)_{n\in\mathbb{N}}\subseteq\mathcal{F}$ be a sequence such that
$
	\sup_{f\in\mathcal{F}}|\mu(f)-\nu(f)|=
	\lim_{n\to\infty}|\mu(f_n)-\nu(f_n)|.
$
For each $f_n$, choose $g_n\in\mathcal{F}_0$ such that
$\mu(|f_n-g_n|)+\nu(|f_n-g_n|)\le n^{-1}$.  Then
$$
	\sup_{f\in\mathcal{F}}|\mu(f)-\nu(f)|=
	\lim_{n\to\infty}|\mu(g_n)-\nu(g_n)|\le
	\sup_{f\in\mathcal{F}_0}|\mu(f)-\nu(f)|,
$$
which clearly yields the result (as $\mathcal{F}_0\subseteq\mathcal{F}$).  
In the case of random probability measures, we simply apply the nonrandom 
result pointwise.
\end{proof}

To prove $2\Rightarrow 8$ we use the ergodic decomposition (cf.\ 
Appendix \ref{app:decomp}). Consider a stationary sequence 
$(Z_n)_{n\in\mathbb{N}}$ of $X$-valued random 
variables on a probability space $(\Omega,\mathcal{G},\mathbf{P})$.  Using 
Corollary \ref{cor:countdens} and the ergodic theorem, it suffices to 
prove that
$$
	\mathbf{P}\left[
	\limsup_{n\to\infty}
	\sup_{f\in\mathcal{F}_0}\left|
	\frac{1}{n}\sum_{k=1}^nf(Z_k)-
	\limsup_{N\to\infty}
	\frac{1}{N}\sum_{k=1}^N f(Z_k)
	\right|=0
	\right]=1.
$$
The event inside the probability is an 
$\mathcal{X}^{\otimes\mathbb{N}}$-measurable function of
$(Z_n)_{n\in\mathbb{N}}$.  Therefore, by Theorem 
\ref{thm:ergdecomp} in Appendix \ref{app:decomp}, it suffices to prove
the result for the case that $(Z_n)_{n\in\mathbb{N}}$ is 
stationary and ergodic.  But in the ergodic case
$\frac{1}{N}\sum_{k=1}^N f(Z_k)\to\mathbf{E}(f(Z_0))$ a.s., 
so that the result follows from the Blum-DeHardt argument.

To prove the implication $2\Rightarrow 7$, we aim to repeat the proof of 
$2\Rightarrow 8$ with a suitable tail decomposition (cf.\ Theorem
\ref{thm:taildecomp} in Appendix \ref{app:decomp}).  On an 
underlying probability space $(\Omega,\mathcal{G},\mathbf{P})$, let 
$(\mathcal{G}_{-n})_{n\in\mathbb{N}}$ be a reverse filtration such that 
$\mathcal{G}_{-n}\subseteq\mathcal{G}$ is countably generated for each 
$n\in\mathbb{N}$, and consider a random variable $Z$ taking values in the 
standard space $(X,\mathcal{X})$.
Using Corollary \ref{cor:countdens} and the reverse
martingale convergence theorem, it evidently suffices to prove that
$$
	\mathbf{P}\left[
	\limsup_{n\to\infty}
	\sup_{f\in\mathcal{F}_0}\left|
	\mathbf{E}(f(Z)|\mathcal{G}_{-n})-
	\limsup_{N\to\infty}
	\mathbf{E}(f(Z)|\mathcal{G}_{-N})
	\right|=0
	\right]=1.
$$
If $(\Omega,\mathcal{G})$ is standard, then by Theorem 
\ref{thm:taildecomp} it suffices to 
prove the result for the case that the tail $\sigma$-field 
$\mathcal{G}_{-\infty}=\bigcap_n\mathcal{G}_{-n}$ is trivial.
But in that case
$\mathbf{E}(f(Z)|\mathcal{G}_{-n})\to\mathbf{E}(f(Z))$ a.s.,
so that the result follows from the Blum-DeHardt argument.

It therefore remains to show that there is no loss of generality in 
assuming that $(\Omega,\mathcal{G})$ is standard.  To this end, choose for 
every $n\ge 1$ a countable generating class 
$(H_{n,j})_{j\in\mathbb{N}}\subseteq\mathcal{G}_{-n}$, and define the 
$\{0,1\}^\mathbb{N}$-valued random variable 
$Z_{-n}=(\mathbf{1}_{H_{n,j}})_{j\in\mathbb{N}}$.  Then, by construction, 
$\mathcal{G}_{-n}=\sigma\{Z_{-k}:k\ge n\}$.  If we define $Z_0=Z$, then it 
is clear that the implication $2\Rightarrow 7$ depends only on the law of 
$(Z_{-n})_{n\ge 0}$.  There is therefore no loss of generality in assuming 
that $(\Omega,\mathcal{G})$ is the canonical space of the process 
$(Z_{-n})_{n\ge 0}$, which is clearly standard as $\{0,1\}^{\mathbb{N}}$ 
is Polish.

\subsection{$\{6,7,8\}\Rightarrow 1$}

These implications follow from the fact that each of the conditions 
$\{6,7,8\}$ contains condition $1$ as a special case.  For the implication 
$6\Rightarrow 1$, it suffices to choose $\mu_n$ to be the empirical 
measure of an i.i.d.\ sequence with distribution $\mu$. Similarly, the 
implication $8\Rightarrow 1$ follows from the fact that an i.i.d.\ 
sequence is stationary and ergodic.  Finally, the implication
$7\Rightarrow 1$ follows from the following well known construction.
Let $(X_k)_{k\in\mathbb{N}}$ be an i.i.d.\ sequence of $X$-valued
random variables with distribution $\mu$, let $Z=X_1$, and let 
$\mathcal{G}_{-n}=\sigma\{\sum_{k=1}^n\mathbf{1}_A(X_k):A\in\mathcal{X}\}$.  
As $(X,\mathcal{X})$ is standard, $\mathcal{X}$ and hence 
$\mathcal{G}_{-n}$ are countably generated.  Moreover, we have 
$$
	\mathbf{E}(f(Z)|\mathcal{G}_{-n})=
	\mathbf{E}(f(X_\ell)|\mathcal{G}_{-n})=
	\frac{1}{n}\sum_{k=1}^n\mathbf{E}(f(X_k)|\mathcal{G}_{-n})=
	\frac{1}{n}\sum_{k=1}^nf(X_k)
$$ 
for any bounded measurable function $f$ and $1\le\ell\le n$, 
as the right hand side is $\mathcal{G}_{-n}$-measurable and every element 
of $\mathcal{G}_{-n}$ is symmetric under permutations of 
$\{X_1,\ldots,X_n\}$.  Therefore, $\frac{1}{n}\sum_{k=1}^n\delta_{X_k}$ is 
a version of the regular conditional probability 
$\mathbf{P}(Z\in\,\cdot\,|\mathcal{G}_{-n})$ for every $n\ge 1$. By the 
law of large numbers and the martingale convergence theorem, it follows 
that $\mu$ is a version of the regular conditional probability 
$\mathbf{P}(Z\in\,\cdot\,|\mathcal{G}_{-\infty})$.  The implication 
$7\Rightarrow 1$ is now immediate.

\section{Proof of Proposition \ref{prop:vc}}
\label{sec:vc}

The construction of the class $\mathcal{C}$ in Proposition \ref{prop:vc} 
is based on a combinatorial construction due to Alon, Haussler, and Welzl 
\cite[Theorem A(2)]{AHW87}.  We begin by recalling the essential results 
in that paper, and then proceed to the proof of Proposition \ref{prop:vc}.

\subsection{Construction}

Let $q\ge 2$ be a prime number, and denote by $\mathbb{F}_q$ the finite 
field $\mathbb{Z}/q\mathbb{Z}$ of order $q$.  In the following, we 
consider the three-dimensional vector space $\mathbb{F}_q^3$ over the 
finite field $\mathbb{F}_q$.  Denote by $V_q$ the family of all 
one-dimensional subspaces of $\mathbb{F}_q^3$, and denote by $E_q$ the 
family of all two-dimensional subspaces of $\mathbb{F}_q^3$. Each element 
of $E_q$ is identified with a subset of $V_q$ by inclusion, that is,
a two-dimensional subspace $C\in E_q$ is identified with the set of 
one-dimensional subspaces $x\in V_q$ contained in it. An elementary 
counting argument, cf.\ \cite[section 9.3]{Cam94}, yields the 
following properties:
\begin{enumerate}
\item $\card V_q=\card E_q=q^2+q+1$.
\item Every set $C\in E_q$ contains exactly $q+1$ points in $V_q$.
\item Every point $x\in V_q$ belongs to exactly $q+1$ sets in $E_q$.
\item For every $x,x'\in V_q$, $x\ne x'$ there is a unique set
$C\in E_q$ with $x,x'\in C$.
\end{enumerate}
A pair $(V_q,E_q)$ with these properties is called a \emph{finite 
projective plane} of order $q$.  For our purposes, the key property
of finite projective planes is the following result due to
Alon, Haussler, and Welzl, whose proof is given in \cite[p.\ 336]{AHW87}
(the proof is based on a combinatorial lemma proved in 
\cite[Theorem 2.1(2)]{Alon85}).

\begin{prop}
\label{prop:alon}
Let $q\ge 2$ be prime, define $m=q^2+q+1$, and let $\varepsilon>0$.
Then for any partition $\pi$ of $V_q$ such that
$(\card\pi)^2\le m^{1/2}(1-\varepsilon)$, we have
$$
	\max_{C\in E_q}
	\frac{\card\partial_\pi C}{m}>\varepsilon.
$$
Here we defined the $\pi$-boundary
$\partial_\pi C:=\bigcup\{P\in\pi:P\cap C\ne\varnothing\mbox{ and }
P\not\subseteq C\}$.
\end{prop}

We now proceed to construct the class $\mathcal{C}$ in Proposition 
\ref{prop:vc}.  Let $q_j\uparrow\infty$ be an increasing sequence of 
primes ($q_j\ge 2$), and define $m_j=q_j^2+q_j+1$.  We now partition
$\mathbb{N}$ into consecutive blocks of length $m_j$, as follows:
$$
	\mathbb{N} = \bigcup_{j=1}^\infty N_j,\qquad
	N_j=\left\{
	\sum_{i=1}^{j-1}m_i+1,\ldots,\sum_{i=1}^jm_i
	\right\}\simeq V_{q_j}.
$$
Define $\mathcal{C}$ as the disjoint union of copies of $E_{q_j}$
defined on the blocks $N_j$: that is, choose for every $j$ a 
bijection $\iota_j:V_{q_j}\to N_j$, and define
$$
	\mathcal{C}=\bigcup_{j=1}^\infty\mathcal{C}_j,\qquad
	\mathcal{C}_j = \{ B\subseteq N_j :
	\iota_j^{-1}(B)\in E_{q_j}\}.
$$
We claim that the countable class $\mathcal{C}$ of subsets of $\mathbb{N}$ 
has $\gamma$-dimension two.  

\begin{lem}
$\mathcal{C}$ has Vapnik-Chervonenkis dimension two.
\end{lem}

\begin{proof}
Choose any three distinct points $n_1,n_2,n_3\in\mathbb{N}$.  If two of 
these points are in distinct intervals $N_j$, then no set in $\mathcal{C}$ 
contains both points.  On the other hand, suppose that all three points 
are in the same interval $N_j$.  Then by the properties of the finite 
projective plane, either there is no set in $\mathcal{C}$ that contains 
all three points, or there is no set that contains two of the points but 
not the third (as each pair of points must lie in a unique set in 
$\mathcal{C}$).  Thus we have shown that no family of three points 
$\{n_1,n_2,n_3\}$ is $\gamma$-shattered for $0<\gamma<1$.  On the other 
hand, it is easily seen that the properties of the finite projective 
plane imply that any pair of points $\{n_1,n_2\}$ belonging to the same 
interval $N_j$ is $\gamma$-shattered for $0<\gamma<1$.
\end{proof}

\subsection{Proof of Proposition \ref{prop:vc}}

The following crude lemma yields lower bounds on the bracketing numbers.

\begin{lem}
\label{lem:crude}
Let $\mu$ be a probability measure on $\mathbb{N}$.  Then
$$
	\inf_{\card\pi\le 3^N}\sup_{C\in\mathcal{C}}
	\mu(\partial_\pi C)>\varepsilon
	\qquad\mbox{implies}\qquad
	N_{[]}(\mathcal{C},\varepsilon,\mu)>N,
$$
where the infimum ranges over all partitions of $\mathbb{N}$ with
$\card\pi\le 3^N$.
\end{lem}

\begin{proof}
Suppose $N_{[]}(\mathcal{C},\varepsilon,\mu)\le N$.  Then there
are $k\le N$ pairs $\{C_i^+,C_i^-\}_{i\le k}$ of subsets of 
$\mathbb{N}$ such that $\mu(C_i^+\backslash C_i^-)\le\varepsilon$
for all $1\le i\le k$, and for every $C\in\mathcal{C}$, there exists
$1\le i\le k$ such that $C_i^-\subseteq X\subseteq C_i^+$.
Let $\pi$ be the partition generated by $\{C_i^+,C_i^-:1\le i\le k\}$.
Then $\card\pi\le 3^N$, as $\pi$ is the common refinement of at most $N$ 
partitions $\{C_i^-,C_i^+\backslash C_i^-,\mathbb{N}\backslash C_i^+\}$
of size three.  

Now choose any $C\in\mathcal{C}$, and choose $1\le i\le k$ such that 
$C_i^-\subseteq C\subseteq C_i^+$.  As $C_i^-$ and $\mathbb{N}\backslash 
C_i^+$ are unions of atoms of $\pi$ by construction, and as 
$C_i^-\subseteq C$ and $(\mathbb{N}\backslash C_i^+)\cap C=\varnothing$, 
we evidently have $\partial_\pi C \subseteq C_i^+\backslash C_i^-$.  Thus 
$\mu(\partial_\pi C)\le \varepsilon$.  As this holds for any 
$C\in\mathcal{C}$, we complete the proof by contradiction. 
\end{proof}

Denote by $\mu_j$ the uniform distribution on $N_j$.
Let $(p_j)_{j\in\mathbb{N}}$ be a sequence of nonnegative numbers
$p_j\ge 0$ so that $\sum_jp_j=1$, and define the probability measure
$$
	\mu = \sum_{j=1}^\infty p_j\mu_j.
$$
We first obtain a lower bound on $N_{[]}(\mathcal{C},\varepsilon,\mu)$.
Subsequently, we will be able to choose the sequence 
$(p_j)_{j\in\mathbb{N}}$ such that this bound grows arbitrarily quickly.

To obtain a lower bound, let us suppose that 
$N_{[]}(\mathcal{C},\varepsilon,\mu)\le N$.  Then applying Lemma 
\ref{lem:crude}, there exists a partition $\pi$ of $\mathbb{N}$ 
with $\card\pi\le 3^N$ such that
$$
	\sup_{j\in\mathbb{N}}p_j\min_{\card\pi'\le 3^N}
	\max_{C\in E_{q_j}}\frac{\card\partial_{\pi'}C}{m_j}
	\le
	\sup_{j\in\mathbb{N}}p_j\max_{C\in\mathcal{C}_j}\mu_j(\partial_\pi C)
	\le\sup_{C\in\mathcal{C}}\mu(\partial_\pi C)
	\le\varepsilon.
$$
By Proposition \ref{prop:alon},
$$
	\min_{\card\pi'\le 3^N}
        \max_{C\in E_{q_j}}\frac{\card\partial_{\pi'}C}{m_j}
	\le\frac{\varepsilon}{p_j}\quad\mbox{implies}\quad
	m_j^{1/4}\sqrt{1-\frac{\varepsilon}{p_j}\wedge 1}<3^{N}.
$$
Therefore, $N_{[]}(\mathcal{C},\varepsilon,\mu)\le N$ implies that
$$
	N>\frac{1}{4}\log_3 m_j
	+ \frac{1}{2}\log_3\left(1-\frac{\varepsilon}{p_j}\wedge 1\right)
$$
for every $j\in\mathbb{N}$.  It follows that
$$
	N_{[]}(\mathcal{C},\varepsilon,\mu)\ge\sup_{j\in\mathbb{N}}
	\left\lfloor
	\frac{1}{4}\log_3 m_j
	+ \frac{1}{2}\log_3\left(1-\frac{\varepsilon}{p_j}\wedge 1\right)
	\right\rfloor.
$$
This bound holds for any choice of $(p_j)_{j\in\mathbb{N}}$.

Fix $n(\varepsilon)\uparrow\infty$ as $\varepsilon\downarrow 0$.
We now choose $(p_j)_{j\in\mathbb{N}}$ such that
$N_{[]}(\mathcal{C},\varepsilon,\mu)\ge n(\varepsilon)$.
First, as $m_j\uparrow\infty$, we can choose a subsequence
$j(k)\uparrow\infty$ such that
$$
	m_{j(\lfloor \log_2(2/3\varepsilon)\rfloor)} \ge
	3^{4n(\varepsilon)+6}\qquad\mbox{for all }0<\varepsilon<1/3.	
$$
Now define $(p_j)_{j\in\mathbb{N}}$ as follows:
$$
	p_{j(k)}=2^{-k}\quad\mbox{for }k\in\mathbb{N},\qquad
	p_j=0\quad\mbox{for }j\not\in\{j(k):k\in\mathbb{N}\}.
$$
Then we clearly have, setting $J(\varepsilon)
=j(\lfloor \log_2(2/3\varepsilon)\rfloor)$,
$$
	N_{[]}(\mathcal{C},\varepsilon,\mu)\ge
	\left\lfloor
	\frac{1}{4}\log_3 m_{J(\varepsilon)}
	+ \frac{1}{2}\log_3\left(1-\frac{\varepsilon}{p_{J(\varepsilon)}
	}\wedge 1\right)
	\right\rfloor \ge
	\left\lfloor
	n(\varepsilon)+1
	\right\rfloor \ge n(\varepsilon)
$$
for all $0<\varepsilon<1/3$.   This completes the proof.

\begin{appendices}

\section{Boolean and stochastic independence}
\label{app:boole}

An essential property of a Boolean $\sigma$-independent sequence of sets 
is that there must exist a probability measure under which these sets are 
i.i.d.  This idea dates back to Marczewski \cite{Mar48}, who showed that 
such a probability measure exists on the $\sigma$-field generated by these 
sets.  For our purposes, we will need the resulting probability measure to 
be defined on the larger $\sigma$-field $\mathcal{X}$ of the underlying 
standard measurable space $(X,\mathcal{X})$.  One could apply an extension 
theorem for measures on standard measurable spaces (for example, \cite[p.\ 
194]{Var63}) to deduce the existence of such a measure from Marczewski's 
result.  However, a direct proof is easily given.

\begin{thm}
\label{thm:marczewski}
Let $(X,\mathcal{X})$ be a standard measurable space. Let
$(A_i,B_i)_{i\in\mathbb{N}}$ be a sequence of pairs of sets
$A_i,B_i\in\mathcal{X}$ such that $A_i\cap B_i=\varnothing$ for
every $i\in\mathbb{N}$ and
$$
	\bigcap_{j\in F}A_j\cap\bigcap_{j\not\in F}B_j
	\ne\varnothing\quad\mbox{for every }F\subseteq\mathbb{N}.
$$
Let $p\in[0,1]$.  Then there exists a probability measure $\mu$ on
$(X,\mathcal{X})$ such that $\mu(A_i)=\mu(X\backslash B_i)=p$
for every $i\in\mathbb{N}$, and such that $(A_i)_{i\in\mathbb{N}}$
are independent under $\mu$.
\end{thm}

\begin{proof}
Let $\mathcal{B}^*$ be the universal completion of the 
the Borel $\sigma$-field of $\{0,1\}^{\mathbb{N}}$, and let
$C_j=\{\omega\in\{0,1\}^{\mathbb{N}}:\omega_j=1\}$ for 
$j\in\mathbb{N}$. Moreover, let $\nu$ be the probability measure on 
$\mathcal{B}^*$ under which $(C_j)_{j\in\mathbb{N}}$ are independent and 
$\nu(C_j)=p$ for every $j\in\mathbb{N}$.

Define for every $\omega\in\{0,1\}^{\mathbb{N}}$ the set
$$
	H(\omega) = 
	\bigcap_{j:\omega_j=1}A_j\cap\bigcap_{j:\omega_j=0}B_j.
$$ 
It suffices to show that there is a measurable map 
$\iota:(\{0,1\}^{\mathbb{N}},\mathcal{B}^*) \to(X,\mathcal{X})$ such that 
$\iota(\omega)\in H(\omega)$ for every $\omega\in\{0,1\}^{\mathbb{N}}$.  
Indeed, as $\iota^{-1}(A_j)=C_j$ and 
$\iota^{-1}(B_j)=\{0,1\}^{\mathbb{N}}\backslash C_j$ for every 
$j\in\mathbb{N}$, the measure $\mu(\cdot)=\nu(\iota^{-1}(\cdot))$ has the 
desired properties. 

It remains to prove the existence of $\iota$.  To this end, note that
the set
$$
	\Gamma=\{(\omega,x):x\in H(\omega)\} =
	\bigcap_{j\in\mathbb{N}}\left\{
	C_j\times A_j\cup 
	\left(\{0,1\}^{\mathbb{N}}\backslash C_j\right)\times B_j
	\right\}
$$
is measurable
$\Gamma\in\mathcal{B}(\{0,1\}^{\mathbb{N}})\otimes\mathcal{X}$,
where $\mathcal{B}(\{0,1\}^{\mathbb{N}})$ denotes the Borel 
$\sigma$-field of $\{0,1\}^{\mathbb{N}}$.  As $H(\omega)$ is nonempty
for every $\omega\in\{0,1\}^{\mathbb{N}}$ by assumption, the
existence of $\iota$ now follows by the measurable section theorem
\cite[Theorem 8.5.3]{Cohn80}.
\end{proof}

\begin{rem}
In the above proof, the assumption that $(X,\mathcal{X})$ is standard
is required to apply the measurable section theorem.  When 
$(X,\mathcal{X})$ is an arbitrary measurable space, we could of course
invoke the axiom of choice to obtain a map $\iota:\{0,1\}^{\mathbb{N}}
\to X$ such that $\iota(\omega)\in H(\omega)$ for every 
$\omega\in\{0,1\}^{\mathbb{N}}$, but such a map need not be measurable
in general.
On the other hand, as $\iota^{-1}(A_j)=C_j$ and 
$\iota^{-1}(B_j)=\{0,1\}^{\mathbb{N}}\backslash C_j$, it follows
that $\iota$ is necessarily Borel-measurable if we choose
$\mathcal{X}=\sigma\{A_j,B_j:j\in\mathbb{N}\}$.  Thus we recover
a result along the lines of Marczewski by using the same proof.
\end{rem}

The proof of Theorem \ref{thm:scales} uses the following connection 
between Boolean independence and $\gamma$-shattering which is a trivial
modification of a result of Assouad \cite{Ass83} (cf.\ \cite[Theorem 
4.6.2]{Dud99}).  We give the proof for completeness.

\begin{lem}
\label{lem:assouad}
Let $\{f_1,\ldots,f_{2^n}\}$ be a finite family of functions 
on a set $X$ that is Boolean independent at levels $(\alpha,\beta)$
with $\beta-\alpha\ge\gamma$.  Then the family $\{f_1,\ldots,f_{2^n}\}$
$\gamma$-shatters some finite subset $\{x_1,\ldots,x_n\}\subseteq X$.
\end{lem}

\begin{proof}
Define $\ell(F)=1+\sum_{j\in F}2^{j-1}$ for $F\subseteq\{1,\ldots,n\}$,
so that $\ell(F)$ assigns to every $F\subseteq\{1,\ldots,n\}$ a unique
integer between $1$ and $2^n$.  Choose some point
$$
	x_j\in \bigcap_{F\ni j}\{f_{\ell(F)}<\alpha\}\cap
	\bigcap_{F\not\ni j}\{f_{\ell(F)}>\beta\}
$$
for every $j=1,\ldots,n$.
Then for any $F\subseteq\{1,\ldots,n\}$, we have
$f_{\ell(F)}(x_j)<\alpha$ if $j\in F$ and 
$f_{\ell(F)}(x_j)>\beta$ if $j\not\in F$.
Therefore $\{x_1,\ldots,x_n\}$ is $\gamma$-shattered.
\end{proof}

\section{Decomposition theorems}
\label{app:decomp}

Part of the proof of Corollary \ref{cor:uniformity} relies on the
decomposition of stochastic processes with respect to
the invariant and tail $\sigma$-fields.  These theorems will be
given presently.

The first theorem is the well-known ergodic decomposition.  As this result 
is classical, we state it here without proof (see \cite[Theorem 
6.6]{Var01} or \cite[Theorem 10.26]{Kal02}, for example, for elementary 
proofs).  In the following, for any standard space $(Y,\mathcal{Y})$, we 
denote by $\mathcal{P}(Y,\mathcal{Y})$ the space of probability measures 
on $(Y,\mathcal{Y})$.  The space $\mathcal{P}(Y,\mathcal{Y})$ is endowed 
with the $\sigma$-field generated by the evaluation mappings 
$\pi_B:\mu\mapsto\mu(B)$, $B\in\mathcal{Y}$.  Recall that if 
$(X,\mathcal{X})$ is standard, then so is 
$(X^{\mathbb{N}},\mathcal{X}^{\otimes\mathbb{N}})$.

\begin{thm}
\label{thm:ergdecomp}
Let $(X,\mathcal{X})$ be a standard space, and denote by
$(Z_n)_{n\in\mathbb{N}}$ the canonical process on the space
$(X^{\mathbb{N}},\mathcal{X}^{\otimes\mathbb{N}})$.
Let $\mu\in\mathcal{P}(X^{\mathbb{N}},\mathcal{X}^{\otimes\mathbb{N}})$
be a stationary probability measure.  Then there exists a probability 
measure $\rho$ on  
$\mathcal{P}(X^{\mathbb{N}},\mathcal{X}^{\otimes\mathbb{N}})$ such that
$$
	\mu(A) = \int \nu(A)\,\rho(d\nu)
	\quad\mbox{for every }
	A\in\mathcal{X}^{\otimes\mathbb{N}},
$$
and such that there exists a measurable subset $B$ of\/
$\mathcal{P}(X^{\mathbb{N}},\mathcal{X}^{\otimes\mathbb{N}})$
with $\rho(B)=1$ and with the property that
every $\nu\in B$ is stationary and ergodic.
\end{thm}

The second theorem is similar in spirit to Theorem \ref{thm:ergdecomp}, 
where we now decompose with respect to the tail $\sigma$-field rather than 
with respect to the invariant $\sigma$-field.  This result is closely 
related to the decomposition theorem for Gibbs measures (see, for example, 
\cite{Dyn78}).  For completeness, we provide a self-contained proof.

\begin{thm}
\label{thm:taildecomp}
Let $(\Omega,\mathcal{G},\mu)$ be a standard probability space. Let
$(\mathcal{G}_{-n})_{n\in\mathbb{N}}$ be a reverse filtration
with each $\mathcal{G}_{-n}\subseteq\mathcal{G}$ countably
generated.  Fix for every $n\in\mathbb{N}$ a version $\mu_{-n}$ of the 
regular conditional probability $\mu(\,\cdot\,|\mathcal{G}_{-n})$. Then 
there exists a probability measure $\rho$ on
$\mathcal{P}(\Omega,\mathcal{G})$ such that
$$
	\mu(A) = \int \nu(A)\,\rho(d\nu)
	\quad\mbox{for every }
	A\in\mathcal{G},
$$
and such that there is a measurable subset $B$ of\/
$\mathcal{P}(\Omega,\mathcal{G})$ with $\rho(B)=1$ and
\begin{enumerate}
\item 
The tail $\sigma$-field $\mathcal{G}_{-\infty}=\bigcap_n\mathcal{G}_{-n}$
is $\nu$-trivial for every $\nu\in B$.
\item $\nu(A|\mathcal{G}_{-n})=\mu_{-n}(A)$ $\nu$-a.s.\ for every
$\nu\in B$, $A\in\mathcal{G}$, and $n\in\mathbb{N}$.
\end{enumerate}
\end{thm}

\begin{proof}
Let $\mu_{-\infty}$ be a version of the regular conditional probability 
$\mu(\,\cdot\,|\mathcal{G}_{-\infty})$, whose existence 
is guaranteed as $(\Omega,\mathcal{G})$ is standard.  We consider
$\mu_{-\infty}:\Omega\to\mathcal{P}(\Omega,\mathcal{G})$ as a 
$\mathcal{G}_{-\infty}$-measurable random probability measure 
$\omega\mapsto\mu_{-\infty}^\omega$ in the usual manner (e.g., 
\cite[Lemma 1.40]{Kal02}).  Let 
$\rho\in\mathcal{P}(\mathcal{P}(\Omega,\mathcal{G}))$ be the 
law under $\mu$ of the random measure $\mu_{-\infty}$.  It follows 
directly from the definition of regular conditional probability that
$$
	\mu(A) = \int \mu_{-\infty}^\omega(A)\,\mu(d\omega) =
	\int \nu(A)\,\rho(d\nu)
	\quad\mbox{for every }
	A\in\mathcal{G}.
$$
It remains to obtain a set $B$ with the two properties in the
statement of the theorem.

We begin with the second property.  Note that
\begin{multline*}
	\int |\nu(\mathbf{1}_C\mu_{-n}(A))-\nu(A\cap C)|\,\rho(d\nu)=\\
	\int |\mu(\mathbf{1}_C\mu(A|\mathcal{G}_{-n})|\mathcal{G}_{-\infty})
	-\mu(A\cap C|\mathcal{G}_{-\infty})|\,d\mu = 0
\end{multline*}
for every $n\in\mathbb{N}$, $A\in\mathcal{G}$,
and $C\in\mathcal{G}_{-n}$.  Let $\mathcal{G}_{-n}^0$ be a countable 
generating algebra for $\mathcal{G}_{-n}$ and let $\mathcal{G}^0$ 
be a countable generating algebra for $\mathcal{G}$.  Evidently
$$
	\int \mathbf{1}_C(\omega)\,\mu_{-n}^\omega(A)\,\nu(d\omega) =
	\nu(A\cap C)
	\quad\mbox{for every }n\in\mathbb{N},~
	A\in\mathcal{G}^0,~C\in\mathcal{G}_{-n}^0
$$
for all $\nu$ in a measurable subset $B_0$ of 
$\mathcal{P}(\Omega,\mathcal{G})$ with
$\rho(B_0)=1$.  But the monotone class theorem allows to extend this
identity to all $A\in\mathcal{G}$ and
$C\in\mathcal{G}_{-n}$.  Thus we have
$\nu(A|\mathcal{G}_{-n})=\mu_{-n}(A)$ $\nu$-a.s.\ for every
$\nu\in B_0$, $A\in\mathcal{G}$, and $n\in\mathbb{N}$.

We now proceed to the first property.  For any
$A\in\mathcal{G}$, we have
\begin{multline*}
	\int \nu(\nu(A|\mathcal{G}_{-\infty})=\nu(A))\,\rho(d\nu) =
	\int 
	\nu\bigg(\limsup_{n\to\infty}\mu_{-n}(A)=\nu(A)\bigg)\,\rho(d\nu) =
	\\
	\mu\bigg(\limsup_{n\to\infty}\mu_{-n}(A)=
	\mu(A|\mathcal{G}_{-\infty})\bigg)=1,
\end{multline*}
where we have used the martingale convergence theorem and the previously
established fact that $\nu(\mu_{-n}(A)=\nu(A|\mathcal{G}_{-n})
~\mbox{for all }n\in\mathbb{N})=1$ for $\rho$-a.e.\ $\nu$.
Therefore, it follows that $\nu(A|\mathcal{G}_{-\infty})=\nu(A)$ 
$\nu$-a.s.\ for all $A\in\mathcal{G}^0$ 
for every $\nu$ in a measurable
subset $B_1$ of $\mathcal{P}(\Omega,\mathcal{G})$
with $\rho(B_1)=1$.  By the monotone class theorem
$\nu(A|\mathcal{G}_{-\infty})=\nu(A)$ $\nu$-a.s.\ for every $\nu\in B_1$
and $A\in \mathcal{G}$.  But then evidently 
$\mathcal{G}_{-\infty}$ is $\nu$-trivial for every $\nu\in B_1$.  Choosing 
$B=B_0\cap B_1$ completes the proof.
\end{proof}

\section{Counterexamples in nonstandard spaces}
\label{sec:counter}

The assumption that $(X,\mathcal{X})$ is standard is used in the proof of 
Theorem \ref{thm:main} to establish the implications $1,3\Rightarrow 4$ 
and $4\Rightarrow 2$.  The goal of this appendix is to show that these 
implications may indeed fail when $(X,\mathcal{X})$ is not standard.  To 
this end we provide two counterexamples, based on the following simple 
observation.

\begin{lem}
\label{lem:aleph}
There exists a Boolean $\sigma$-independent sequence of functions on
a set $X$ if and only if $\card X\ge 2^{\aleph_0}$.
\end{lem}

\begin{proof}
Suppose there exists a Boolean $\sigma$-independent sequence
$(f_j)_{j\in\mathbb{N}}$ of functions $f_j:X\to\mathbb{R}$.
Then there exist $\alpha<\beta$ such that for every
$F\subseteq\mathbb{N}$, the set
$$
	\bigcap_{j\in F}\{f_j<\alpha\}\cap\bigcap_{j\not\in F}
	\{f_j>\beta\}
$$
contains at least one point.  As these sets are disjoint for distinct
$F\subseteq\mathbb{N}$, and there are $2^{\aleph_0}$ subsets of 
$\mathbb{N}$, it follows that $\card X\ge 2^{\aleph_0}$.  Conversely, if 
$\card X\ge 2^{\aleph_0}$, there exists an injective map 
$\iota:\{0,1\}^{\mathbb{N}}\to X$.  Define the sets
$C_j=\{\iota(\omega):\omega\in\{0,1\}^{\mathbb{N}},~
\omega_j=1\}\subset X$.  Then the sequence
$(\mathbf{1}_{C_j})_{j\in\mathbb{N}}$ is Boolean $\sigma$-independent.
\end{proof}

Both examples below are consistent with the usual axioms of set theory 
(that is, the set theory ZFC) but depend on additional set-theoretic 
axioms.  I do not know whether it is possible to obtain counterexamples in 
the absence of additional axioms.

\subsection{An example where $1,3\not\Rightarrow 4$}

Let $X$ be an uncountable Polish space, and let $\mathcal{X}$ be the 
universal completion of its Borel $\sigma$-field.  Then $(X,\mathcal{X})$ 
is certainly not a standard measurable space. It is known, see 
Sierpi{\'n}ski and Szpilrajn \cite{SS36}, that there exists a set 
$A\in\mathcal{X}$ with $\card A=\aleph_1$ that is \emph{universally null}, 
that is, $\mu(A)=0$ for every nonatomic probability measure $\mu$ on 
$\mathcal{X}$.  As every subset $C\subseteq A$ is in the $\mu$-completion 
of the Borel $\sigma$-field of $X$ for every probability measure 
$\mu$, it follows that $C\in\mathcal{X}$ for every $C\subseteq A$.

As is noted by Dudley, Gin{\'e} and Zinn \cite[p.\ 494]{DGZ91}, the family 
of indicators $\mathcal{F}_A=\{\mathbf{1}_C:C\subseteq A\}$ is a universal 
Glivenko-Cantelli class.  Moreover, as $A$ is a $\mu$-null set for every 
nonatomic probability measure, it is evident that 
$N(\mathcal{F}_A,\varepsilon,\mu)=N(\mathcal{F}_A,\varepsilon,\mu_{\rm 
at})<\infty$ for every $\varepsilon>0$ and 
probability measure $\mu$, where $\mu_{\rm at}$ denotes the atomic part 
of $\mu$.  But assuming the continuum hypothesis, we have 
$\card A=2^{\aleph_0}$ and therefore $\mathcal{F}_A$ contains a Boolean 
$\sigma$-independent sequence $\mathcal{F}$ by Lemma \ref{lem:aleph}. 
Clearly $\mathcal{F}$ is a separable uniformly bounded family of 
measurable functions on $(X,\mathcal{X})$ for which the implications 
$1,3\Rightarrow 4$ of Theorem \ref{thm:main} fail.

\begin{rem} 
The existence of a universally null set does not require the 
continuum hypothesis: Sierpi{\'n}ski and Szpilrajn \cite{SS36} 
construct such a set in ZFC (the construction follows directly
from Hausdorff \cite{Hau36}, see also \cite[Theorem 
1.2]{Lav76}).  Nonetheless, the present counterexample 
does depend on the continuum hypothesis and may fail in its absence.  
Indeed, there exist models of the set theory ZFC in which every 
universally null set has cardinality strictly less than $2^{\aleph_0}$, 
see Laver \cite[p.\ 152]{Lav76}, Miller \cite[pp.\ 577--578]{Mil83}, or 
Ciesielski and Pawlikowski \cite[p.\ xii and Theorem 1.1.4]{CP04}.  In 
such a model, $\mathcal{F}_A$ cannot contain a Boolean 
$\sigma$-independent sequence by Lemma \ref{lem:aleph}.
\end{rem}

\subsection{An example where $4\not\Rightarrow 2$}

The present counterexample follows from the following
result that is proved below.

\begin{prop}
\label{prop:cardindep}
It is consistent with the set theory ZFC that there exists a probability 
space $(X,\mathcal{X},\mu)$ with $\card X < 2^{\aleph_0}$ such that
there is a sequence of sets $(C_j)_{j\in\mathbb{N}}\subset\mathcal{X}$
that are independent under $\mu$ with $\mu(C_j)=1/2$ for every
$j\in\mathbb{N}$.
\end{prop}

This result easily yields the desired example.  Let $(X,\mathcal{X},\mu)$ 
and $(C_j)_{j\in\mathbb{N}}$ be as in Proposition \ref{prop:cardindep}, 
and define the class $\mathcal{F}=\{\mathbf{1}_{C_j}:j\in\mathbb{N}\}$.  
The proof of the implication $3\Rightarrow 4$ of Theorem \ref{thm:main} 
shows that $N_{[]}(\mathcal{F},\varepsilon,\mu)\ge 
N(\mathcal{F},\varepsilon,\mu)=\infty$ for $\varepsilon>0$ sufficiently 
small.  On the other hand, $\mathcal{F}$ cannot contain a Boolean 
$\sigma$-independent sequence by Lemma \ref{lem:aleph}. Thus $\mathcal{F}$ 
is a separable uniformly bounded family of measurable functions on 
$(X,\mathcal{X})$ for which the implication $4\Rightarrow 2$ of Theorem 
\ref{thm:main} fails.

\begin{rem}
It is clear that the present counterexample must depend on a model of set 
theory in which the continuum hypothesis fails.  Indeed, the set $X$ in
Proposition \ref{prop:cardindep} must be uncountable as it supports a 
(stochastically) independent sequence.  Therefore, if we assume the 
continuum hypothesis, then necessarily $\card X \ge 2^{\aleph_0}$ and 
we cannot guarantee the nonexistence of a Boolean $\sigma$-independent 
sequence.
\end{rem}

Denote by $\lambda$ the Lebesgue measure on $[0,1]$, and denote by
$\lambda^*$ the Lebesgue outer measure.  The proof of Proposition 
\ref{prop:cardindep} is based on the following remarkable fact: there 
exist models of the set theory ZFC in which there is a subset 
$X\subset[0,1]$ with $\card X<2^{\aleph_0}$ such that $\lambda^*(X)>0$;
see Martin and Solovay \cite[section 4.1]{MS70}, 
Kunen \cite[Theorem 3.19]{Kun84}, or Judah and Shelah \cite{JS90}.
The existence of such a set $X$ will be assumed in the proof of
Proposition \ref{prop:cardindep}.  Note that the set $X$ cannot be 
Lebesgue measurable (if $X$ were measurable it must contain a Borel set of 
positive measure, which has cardinality $2^{\aleph_0}$ by the Borel 
isomorphism theorem).

\begin{proof}[Proposition \ref{prop:cardindep}]
Assume a model of the set theory ZFC in which there exists a set
$X\subset[0,1]$ with $\card X<2^{\aleph_0}$ such that $\lambda^*(X)>0$.
Let $\mathcal{X}$ be the trace of the Borel $\sigma$-field of $[0,1]$
on $X$, that is, $\mathcal{X}=\{A\cap X:A\in\mathcal{B}([0,1])\}$.
Choose a measurable cover $\tilde X$ of $X$, and note that
$A\cap\tilde X$ is a measurable cover of $A\cap X$ whenever
$A\in\mathcal{B}([0,1])$.  We may therefore unambiguously define
$\mu(A\cap X)=\lambda(A\cap \tilde X)/\lambda(\tilde X)$ for 
$A\in\mathcal{B}([0,1])$, and it is easily verified that $\mu$ is a 
probability measure on $(X,\mathcal{X})$ whose definition does not depend 
on the choice of $\tilde X$.

We now claim the following: for every set $C\in\mathcal{X}$ with
$\mu(C)>0$, there exists a set $C'\in\mathcal{X}$, $C'\subset C$
with $\mu(C')=\mu(C)/2$.  Indeed, let $C=A\cap X$ for some
$A\in\mathcal{B}([0,1])$.  As the function $\phi:t\mapsto
\lambda(A\cap \tilde X\cap [0,t])$ is continuous and
$\phi(0)=0$, $\phi(1)=\lambda(A\cap \tilde X)$, there exists
by the intermediate value theorem
$0<s<1$ such that $\phi(s)=\lambda(A\cap \tilde X)/2$.
Therefore $C'=C\cap[0,s]$ yields the desired set.

Now inductively define for every $n\ge 1$ and
$\omega\in\{0,1\}^n$ a set $A_\omega\in\mathcal{X}$ as follows.
For $n=1$, choose a set $A_0\in\mathcal{X}$ such that
$\mu(A_0)=1/2$, and define $A_1=X\backslash A_0$.
For $n>1$, choose for every $\omega\in\{0,1\}^{n-1}$ a set
$A_{\omega 0}\in\mathcal{X}$ such that $A_{\omega 0}\subset A_{\omega}$
with $\mu(A_{\omega 0})=\mu(A_\omega)/2$, and define
$A_{\omega 1}=A_{\omega}\backslash A_{\omega 0}$.  Finally, define
for every $n\ge 1$ 
$$
	C_n=\bigcup_{\omega\in\{0,1\}^n:\omega_n=0}A_\omega.
$$
Then $\mu(C_n)=1/2$ for every $n\ge 1$, and $\mu(C_{i_1}\cap
\cdots\cap C_{i_k})=2^{-k}$ for every $k\ge 1$ and
$1\le i_1<i_2<\cdots<i_k$.  This evidently completes the proof.
\end{proof}

\end{appendices}

\subsection*{Acknowledgment}
The author would like to thank Terry Adams and Andrew Nobel for making 
available an early version of \cite{AN10b} and for interesting 
discussions on the topic of this paper.

%\bibliographystyle{acmtrans-ims}
%\bibliography{ref}

\end{document}